\documentclass[12pt]{amsart}
\usepackage{amssymb,amsbsy}
\usepackage{amscd}
\usepackage{enumerate}
\usepackage{setspace}
\usepackage{tikz}
\usepackage{bm}
\usepackage{cite}
\usepackage{graphicx}
\usepackage{hyperref}
\usepackage{colortbl}
\usepackage{mathtools}
\allowdisplaybreaks

\usepackage{kpfonts}
\usepackage{stackengine}
\usepackage{calc}
\newlength\shlength
\newcommand\xshlongvec[2][0]{\setlength\shlength{#1pt}%
  \stackengine{-5.6pt}{$#2$}{\smash{$\kern\shlength%
    \stackengine{7.55pt}{$\mathchar"017E$}%
      {\rule{\widthof{$#2$}}{.57pt}\kern.4pt}{O}{r}{F}{F}{L}\kern-\shlength$}}%
      {O}{c}{F}{T}{S}}
  
\usepackage[letterpaper]{geometry}
\theoremstyle{definition}
\newtheorem{theorem}{Theorem}[section]
\newtheorem{lemma}[theorem]{Lemma}
\newtheorem{proposition}[theorem]{Proposition}
\newtheorem{definition}[theorem]{Definition}

\newtheorem{example}[theorem]{Example}

\newcommand{\ovl}[1]{\overline{#1}}	
\newcommand{\parn}[1]{\left( #1 \right)}
\newcommand{\conv}[1]{\mathrm{conv}\parn{#1}}	

\newcommand{\Z}{{\mathbb{Z}}}
\newcommand{\R}{{\mathbb{R}}}
\newcommand{\C}{{\mathbb{C}}}
\newcommand{\N}{{\mathbb{N}}}

\renewcommand{\Re}{\ensuremath{\mathrm{Re}}}
\renewcommand{\Im}{\ensuremath{\mathrm{Im}}}

\title[Numerical range of a tridiagonal operator]{The numerical range of  a  periodic tridiagonal operator reduces to the
 numerical range of a finite matrix}

\date{January 2021}

\author{Benjam\'in A.~Itz\'a-Ortiz}
\author{Rub\'en A.~Mart\'inez-Avenda\~no}
\author{Hiroshi Nakazato}

\address{Centro de Investigaci\'on en Matem\'aticas, Universidad Aut\'onoma del Estado de Hidalgo, Pachuca, Hidalgo, Mexico}
\address{Departamento Acad\'emico de Matem\'aticas, Instituto Tecnol\'ogico Aut\'onomo de M\'exico, Mexico City, Mexico}
\address{Department of Mathematics and Physics, Hirosaki University, Hirosaki City, Japan}

\thanks{The second author's research is partially supported by the Asociaci\'on Mexicana de Cultura A.C.}
\dedicatory{Dedicated to the memory of Rudolf Kippenhahn (1926--2020)}
\begin{document}
\setlength\arraycolsep{2pt}

\begin{abstract}
In this paper we show that the closure of the numerical range of an $n+1$-periodic tridiagonal operator is equal to the numerical range of a $2(n+1) \times   2(n+1)$ complex matrix.
\end{abstract}

\maketitle

\section*{Introduction}

Consider  $\mathcal A$ to be a finite set of complex numbers and let  $a=(a_i)_{i\in\Z}$ be a biinfinite sequence in the total shift space $\mathcal A^\Z$.  In \cite{HI-O2016}, the tridiagonal operator $A_a\colon \ell^2(\Z)\to\ell^2(\Z)$  associated to $a$ is defined as
 \begin{equation}\label{eq:Ab}
 A_{a}= 
\begin{pmatrix}
\ddots & \ddots & & & & & \\
\ddots & 0 & 1 & & & & \\
& a_{-2} & 0 & 1 & & & \\
& & a_{-1} & \framebox[0.4cm][l]{0} & 1 & & \\
& & & a_{0} & 0 & 1 & \\
& & & & a_{1} & 0 & \ddots \\
& & & & & \ddots & \ddots
\end{pmatrix}
\end{equation}
where the square marks the matrix entry at $(0,0)$. In the particular case of the alphabet $\mathcal A=\{-1,1\}$,  the corresponding operator $A_a$ is related to the so called ``hopping sign model'' introduced in \cite{Feinberg} and subsequently studied in many other works, such as \cite{BebianoEtAl,CWChLi10,
CWChLi13,ChD, CN2011,CS, HaggerJFA, Hagger,HI-O2016}, just to name a few. 
On the other hand, when the alphabet is $\mathcal A=\{0,1\}$ some results for computing the  numerical range of $A_a$ are presented 
in \cite{HI-O2016,IM}. In particular, work in \cite{IM}  addresses the case when $a$ is  an $n+1$-periodic  sequence. Relying on the fact that the 
closure of the numerical range of $A_a$ may be written as the closure of the convex hull of an uncountable union of numerical ranges of certain
 matrices, in \cite{IM} the closure of the numerical range of the 2-periodic case is computed by substituting such uncountable union of numerical
 ranges by the convex hull of the union of the numerical ranges of just two $2\times 2$  matrices.
In this work, we further contribute to  the study of  the numerical range of $A_a$ when $a$ is an $n+1$ periodic biinfinite
 sequence. 
 
Instead of working with the operators $A_a$, we work with the more general tridiagonal operators $T=T(a,b,c)$ defined in Section 2, since, as can be
 seen in \cite{IM}, the computation of the closure of the numerical range of $A_a$ is a particular case of that of $T$. Using a result of
 Plaumann and Vinzant \cite{PV}, we show that the closure of the numerical range of the $n+1$ periodic tridiagonal operator $T$ is the numerical 
range of a $2(n+1) \times 2(n+1)$ matrix (cf. Theorem 2.6). 

We divide this work in two sections. In Section~1 we briefly introduce the notation and terminologies needed in the rest of the paper. In Section~2
 we develop the  required machinery, first by computing the Kippenhahn polynomial of the symbol of $n+1$ periodc tridiagonal operatos $T$ on
 $\ell^2(\N_0)$ and then by combining our computations with results of Plaumann and Vizant. We will conclude that the closure of the numerical 
range of $T$ is equal to the numerical range of a $2(n+1)\times 2(n+1)$ matrix $A$. Furthermore, we provide some examples where $A$ can be explicitly computed and we show that the size of $A$ is optimal.

\section{Preliminaries}

In this section we introduce the notation required which will be needed in the following sections. As usual, the symbols $\N$, $\N_0$, $\Z$, $\R$ and $\C$ will denote the set of positive integers, the sets of nonnegative integers, the set of integers, the set of real numbers and the set of complex numbers, respectively.

For a given $n\in\N$, let $a$, $b$ and $c$ be $(n+1)$-periodic infinite sequences in $\mathcal A^{\N_0}$. We will denote by $T=T(a,b,c)$ the  $(n+1)$-periodic tridiagonal operator on $\ell^2(\N_0)$ given by
\[
T= 
\begin{pmatrix}
b_0 & c_0 & & & & & & & &\\
a_1 & b_1 & c_1 & & & & & & & \\
    & a_{2} & b_2 & c_2 & & & & & & \\
    &       & \ddots & \ddots  & \ddots & & & & & \\
    & & & a_{n} & b_n & c_n & & & &\\
    & & & & a_{0} & b_0 & c_0 & & &\\
    & & & & & \ddots & \ddots & \ddots & &\\
    & & & & & &a_{n-1} & b_{n-1}& c_{n-1} & \\
    & & & & & & &  a_n& b_n  & c_n &\\
    & & & & & & & & \ddots &  \ddots&\ddots
\end{pmatrix}.
\]

We should observe that $T$ is a bounded operator since the sum of the moduli of the entries in each column (and in each row) is uniformly bounded (see, e.g., \cite[Example 2.3]{Kato}). The biinfinite matrix $A_a$ is also a bounded operator, as long as the biinfinite sequence $a$ arises from a finite alphabet.

If $n>1$, for each $\phi \in [0, 2\pi)$, following~\cite{BebianoEtAl,IM} we define the symbol of $T$,  as the following  $(n+1)\times (n+1)$ matrix
\begin{equation}\label{eq:symbol_n}
T_{\phi}= 
\begin{pmatrix}
 b_0 & c_0 & 0 & & & 0 & a_0 e^{-i\phi}\\
a_1 & b_1 & c_1 & 0 & & & 0\\
0   & a_{2} & b_2 & c_2 & 0 & & \\
    & \ddots & \ddots &\ddots & \ddots & \ddots   & \\
& & 0 & a_{n-2} & b_{n-2} & c_{n-2} & 0\\
0 & & & 0 & a_{n-1} & b_{n-1} & c_{n-1} \\
c_n e^{i\phi} & 0 & & & 0 & a_n & b_n
\end{pmatrix}; 
\end{equation}
while the symbol  of $T$ for $n=1$ is the $2\times 2$ matrix 
\begin{equation}\label{eq:symbol_2}
    T_{\phi}=
        \begin{pmatrix}
        b_0 & c_0+a_0e^{-i\phi}\\ 
        a_1+c_1 e^{i\phi} & b_1  
        \end{pmatrix}.
\end{equation}

Recall that given a Hilbert space $\mathcal H$ and a bounded operator $A$ on it, the numerical range is defined as the set 
\[
W(A)= \{ \left< A x, x \right> \, : \, \| x \|=1 \}.
\]
The Toeplitz-Haussdorf Theorem establishes that $W(A)$ is a bounded convex subset of $\C$ (closed, if the Hilbert space is finite dimensional) and hence the closure of the numerical range can be seen as the intersection of the closed half-spaces containing the numerical range.

Kippenhahn~\cite{Kip} (see also \cite{ZH}) characterized two vertical support lines of $W(A)$ for a given $n \times n$ matrix as
 $\Re(z)=\lambda_1(A)$ and $\Re(z)=\lambda_n(A)$, where $\lambda_1(A)$ and $\lambda_n(A)$ are the respective largest and least eigenvalues
 of $\Re(A)$ (recall that $\Re(A):=\frac{1}{2} (A + A^*)$ and $\Im(A):=\frac{1}{2i} (A - A^*)$). In fact, if $\alpha \in W(A)$ then 
$\lambda_n(A) \leq \Re(\alpha) \leq \lambda_1(A)$ (and the equalities hold for some points $\alpha_1, \alpha_2 \in W(A)$). Since $e^{i \theta} W(A)=W(e^{i \theta} A)$ for each $\theta \in [0, 2\pi)$, it follows that if  $\alpha \in W(A)$, then $e^{-i\theta} \alpha \in W(e^{-i\theta} A)$ and hence $\Re(e^{-i\theta} \alpha) \leq \lambda_1(e^{-i \theta} A)$. It follows that the lines $\Re(e^{-i\theta} z) = \lambda_1(e^{-i \theta} A)$ are support lines of $W(A)$. Hence the convex set $W(A)$ is uniquely determined by the numbers $\lambda_1(e^{-i\theta} A)$, as $\theta$ varies on the interval $[0, 2\pi)$; i.e. $W(A)$ is determined by the largest eigenvalue of $\Re(e^{-i \theta} A)$, which equals $\cos(\theta) \Re(A) + \sin(\theta) \Im(A)$. Thus the numerical range is determined by the largest roots of the family of characteristic polynomials
\[
\det(t I_n - \cos(\theta) \Re(A) - \sin(\theta) \Im(A)).
\]
The homogeneous polynomial $F_A(t,x,y)=\det(t I_n + x \Re(A) + y \Im(A))$ is called the Kippenhahn polynomial of the matrix $A$. It clearly follows
 that two matrices have the same numerical range if their Kippenhahn polynomials coincide. 
Furthermore, 
\[
\max\{ t \in \R \, : \, F_A(t,-\cos(\theta),-\sin(\theta))=0 \} = \max\{ \Re(e^{-i \theta} z) \, : \, z \in W(A)\}
\]
for each $\theta \in [0,2\pi)$.

\section{The Kippenhahn polynomial of the symbol $T\sb\phi$}

In this section, after some preliminary work, we show that the closure of the numerical range of a $n+1$-periodic tridiagonal operator $T$ is the numerical range of a $2(n+1) \times 2(n+1)$ matrix. 

We will need the following lemma.

\begin{lemma}\label{le:laplace}
Consider the $(n+1)\times(n+1)$ ``almost tridiagonal'' matrix
\[
\Lambda=\begin{pmatrix}
\lambda_{1,1} & \lambda_{1,2} & 0 &  0 & \dots & 0 & 0 & \lambda_{1,n+1} \\
\lambda_{2,1} & \lambda_{2,2} & \lambda_{2,3} & 0 &  \dots & 0 & 0 & 0\\
0 & \lambda_{3,2} & \lambda_{3,3} & \lambda_{3,4} & \dots & 0 & 0 & 0  \\
0 & 0 & \lambda_{4,3} & \lambda_{4,4} & \dots & 0 & 0 & 0 \\
\vdots & \vdots & \vdots & \vdots &\ddots &\vdots &\vdots &\vdots \\
0 & 0 & 0 & 0 & \dots & \lambda_{n-1,n-1} & \lambda_{n-1,n} & 0  \\
0 & 0 & 0 & 0 & \dots & \lambda_{n,n-1} & \lambda_{n,n} & \lambda_{n,n+1}  \\
\lambda_{n+1,1} & 0 & 0 & 0 & \dots & 0 & \lambda_{n+1,n} & \lambda_{n+1,n+1}  
\end{pmatrix},
\]
where every $\lambda_{i,j} \in \C$. Then, $\det(\Lambda)$ equals
\begin{equation*}
    \begin{split}
&
\det\begin{pmatrix}
\lambda_{1,1} & \lambda_{1,2} & 0 &  \dots & 0 & 0 \\
\lambda_{2,1} & \lambda_{2,2} & \lambda_{2,3} &  \dots & 0 & 0\\
0 & \lambda_{3,2} & \lambda_{3,3} &  \dots & 0 & 0  \\
\vdots & \vdots & \vdots &\ddots &\vdots &\vdots \\
0 & 0 & 0 & \dots & \lambda_{n-1,n} & 0  \\
0 & 0 & 0 & \dots & \lambda_{n,n} & \lambda_{n,n+1}  \\
0 & 0 & 0 & \dots & \lambda_{n+1,n} & \lambda_{n+1,n+1}
\end{pmatrix} 
- \lambda_{1,n+1} \lambda_{n+1,1} 
\det \begin{pmatrix}
\lambda_{2,2} & \lambda_{2,3} & \dots & 0 & 0 \\
\lambda_{3,2} & \lambda_{3,3} &  \dots & 0 & 0 \\
\vdots & \vdots &\ddots &\vdots &\vdots  \\
0 & 0 & \dots & \lambda_{n-1,n-1} & \lambda_{n-1,n}  \\
0 & 0 & \dots & \lambda_{n,n-1} & \lambda_{n,n}  
\end{pmatrix} \\
\\
&+ (-1)^{n} \lambda_{n+1,1} \lambda_{1,2} \lambda_{2,3} \cdots \lambda_{n-1,n} \lambda_{n,n+1} 
+ (-1)^{n}\lambda_{1,n+1} \lambda_{2,1} \lambda_{3,2} \cdots \lambda_{n,n-1} \lambda_{n+1,n}.
\end{split}
\end{equation*}
\end{lemma}
\begin{proof}
 This follows by a long (but straightforward) application of the multilinearity of the determinant function and the Laplace Expansion Theorem.
\end{proof}

Let us set the following notation for the rest of this paper. For $0 \leq j < n$ we define
 \begin{equation*}
     \alpha_j=\frac{c_j+\overline{a_{j+1}}}{2},\quad \gamma_j=\frac{c_j-\overline{a_{j+1}}}{2i}
     \end{equation*}
     and
\begin{equation*}     
     \alpha_n=\frac{a_0+\overline{c_n}}{2},\quad
     \gamma_n=\frac{a_0-\overline{c_n}}{2i}.
 \end{equation*}

We now find an expression for the Kippenhahn polynomial $F_{T_\phi}$ of the symbol matrix $T_{\phi}$ of an arbitrary $n+1$-periodic tridiagonal  matrix $T$ acting on $\ell^2(\N_0)$, involving the determinants of some tridiagonal matrices. This expression will be useful in what follows. 

\begin{proposition}\label{prop:kipp}
Let $n\in\N$. Consider the symbol $T\sb\phi$, that is, the $(n+1)\times(n+1)$ matrix defined as in \eqref{eq:symbol_n} for $n\geq 2$ and as in \eqref{eq:symbol_2} for $n=1$. Then the Kippenhahn polynomial of $T_\phi$ is equal to
 \begin{align*}
F_{T_\phi}(t,x,y)={}& G_{n}(t, x, y) - |\alpha_n x +\gamma_n y|^2 H_{n}(t, x, y) 
  + 2(-1)^n \  \Re\left( \left(\overline{\alpha_n} x +\overline{\gamma_n} y\right) {\prod_{j=0}^{n-1}} (\alpha_j x +\gamma_j y)\right) \cos \phi \\ 
  & -2 (-1)^n \  
    \Im\left( \left(\overline{\alpha_n} x +\overline{\gamma_n} y\right) { \prod_{j=0}^{n-1}}  (\alpha_j x +\gamma_j y)\right) \sin \phi,
    \end{align*}
    where $G_n(t,x,y)$ is the determinant of the tridiagonal $(n+1) \times (n+1)$ matrix
     \[
\begin{pmatrix}
    \lambda_{1,1} & \lambda_{1,2} & 0  & 0 &  \ldots &  0 &0 \\
    \lambda_{2,1} & \lambda_{2,2} & \lambda_{2,3} & 0  & \cdots  &0 &  0  \\
    0 & \lambda_{3,2} & \lambda_{3,3} & \lambda_{3,3}  & \cdots  & 0 &  0  \\
    0 & 0 & \lambda_{4,3} & \lambda_{4,4}  & \cdots  &  0  & 0\\
       \vdots & \vdots & \vdots & \vdots & &\vdots &\vdots  \\
       0  &  0  &   0 &  0 & \cdots&    \lambda_{n, n} & \lambda_{n, n+1} \\
       0  &  0  &   0 &  0 & \cdots &    \lambda_{n+1, n} & \lambda_{n+1, n+1} 
\end{pmatrix},
\]
and, where we set $H_n(t,x,y)=1$ when $n=1$, and, for $n\geq 2$, we set $H_n(t,x,y)$ to be the determinant of $(n-1) \times (n-1)$ tridiagonal matrix  
  \[
\begin{pmatrix}
      \lambda_{2,2} & \lambda_{2,3} & 0  &  \cdots & 0 & 0  \\
      \lambda_{3,2} & \lambda_{3,3} & \lambda_{3,4} &  \cdots  & 0 &  0  \\
      0  & \lambda_{4,3} & \lambda_{4,4} & \cdots  &  0 & 0\\
       \vdots & \vdots & \vdots & & \vdots & \vdots \\ 
       0  &  0 & 0  &  \cdots &    \lambda_{n-1, n-1} & \lambda_{n-1, n} \\
              0 & 0 &  0  &   \cdots &    \lambda_{n, n-1} & \lambda_{n, n} 
\end{pmatrix}.
\]
Here we have set, for $1 \leq j \leq n+1$,
\[
\lambda_{j,j}=t +\Re(b_{j-1}) x +\Im(b_{j-1}) y,
\]
and for $ 1 \leq j \leq n$, 
\[
\lambda_{j,j+1}=\alpha_{j-1} x + \gamma_{j-i} y \quad \text{ and} \quad \lambda_{j+1,j}=\ovl{\alpha_{j-1}} x + \ovl{\gamma_{j-1}} y.
\]
\end{proposition}
\begin{proof}
We divide the proof in two cases. For $n+1=2$, by computing the real and imaginary parts of the matrix $T_\phi$ in \eqref{eq:symbol_2}, we obtain that the $2 \times 2$ matrix $t I_2 +x \Re(T_{\phi}) +y \Im(T_{\phi})$ is given by 
\begin{equation*}
\begin{pmatrix} t +\Re(b_0) x +\Im(b_0) y  &  \alpha_0 x +\gamma_0 y +
 (\alpha_1 x +\gamma_1 y) e^{-i \phi} \cr
 (\overline{\alpha_0} x +\overline{\gamma_0} y) + (\overline{\alpha_1} x +\overline{\gamma_1} y) e^{i \phi} & t + \Re(b_1) x +\Im(b_1) y 
  \end{pmatrix},
  \end{equation*}
  where $\alpha_0$, $\alpha_1$, $\gamma_0$ and $\gamma_1$ are as defined above.
The determinant of this matrix can be simplified to
   \begin{align*}
   F_{T_{\phi}}(t, x, y)= {} &(t +\Re(b_0) x +\Im(b_0) y)(t +\Re(b_1) x +\Im(b_1) y) -|\alpha_0 x +\gamma_0 y|^2 - | \alpha_1 x +\gamma_1 y|^2\\
&\quad -  2 \Re\left( (\alpha_0 x +\gamma_0 y)( \overline{\alpha_1} x +\overline{\gamma_1} y) e^{i \phi}\right) \\
= {} &(t +\Re(b_0) x +\Im(b_0) y)(t +\Re(b_1) x +\Im(b_1) y) -|\alpha_0 x +\gamma_0 y|^2 - | \alpha_1 x +\gamma_1 y|^2\\
&\quad -  2 \Re\left( (\alpha_0 x +\gamma_0 y)( \overline{\alpha_1} x +\overline{\gamma_1} y)\right) \cos \phi +  2 \Im \left( (\alpha_0 x +\gamma_0 y)(\overline{\alpha_1} x +\overline{\gamma_1} y)\right)\sin \phi \\
={} & G_1(t,x,y)-|\alpha_1 x + \gamma_1 t|^2 H_1(t,x,y) \\
&\quad -  2 \Re\left( (\alpha_0 x +\gamma_0 y)( \overline{\alpha_1} x +\overline{\gamma_1} y)\right) \cos \phi +  2 \Im \left( (\alpha_0 x +\gamma_0 y)(\overline{\alpha_1} x +\overline{\gamma_1} y)\right)\sin \phi,
\end{align*}
as desired.

Now, for the case $n+1\geq 3$, by computing the real and imaginary parts of the matrix $T_\phi$ in \eqref{eq:symbol_n}, we can observe that $t I_{n+1} + x \Re(T_\phi) + y \Im(T_\phi)$ is the matrix
     \[
\begin{pmatrix}
    \lambda_{1,1} & \lambda_{1,2} & 0  & 0 &  \ldots &  0 & \lambda_{1,n+1} \\
    \lambda_{2,1} & \lambda_{2,2} & \lambda_{2,3} & 0  & \cdots  &0 &  0  \\
    0 & \lambda_{3,2} & \lambda_{3,3} & \lambda_{3,3}  & \cdots  & 0 &  0  \\
    0 & 0 & \lambda_{4,3} & \lambda_{4,4}  & \cdots  &  0  & 0\\
       \vdots & \vdots & \vdots & \vdots & &\vdots &\vdots  \\
       0  &  0  &   0 &  0 & \cdots&    \lambda_{n, n} & \lambda_{n, n+1} \\
       \lambda_{n+1,1}  &  0  &   0 &  0 & \cdots &    \lambda_{n+1, n} & \lambda_{n+1, n+1} 
\end{pmatrix},
\]
where we have now set
\[
\lambda_{1,n+1}=(\alpha_n x + \gamma_n y) e^{-i \phi} \quad \text{and} \quad
\lambda_{n+1,1}=(\ovl{\alpha_n} x + \ovl{\gamma_n} y) e^{i \phi}.
\]
The above matrix is tridiagonal, except for the upper-right and bottom-left corners. 

 We can compute the determinant of the matrix polynomial $t I_{n+1} + x \Re(T_{\phi}) +y \Im(T_{\phi})$ by using Lemma~\ref{le:laplace} obtaining  
  \begin{align*}
  F_{T_\phi}(t,x,y) = {}& \det( t I_{n+1} + x \Re(T_{\phi}) +y \Im(T_{\phi})) \\
={} &  G_n(t,x,y)- |\alpha_n x +\gamma_n y|^2\ H_n(t,x,y) \\ 
  &  +(-1)^n (\overline{\alpha_n} x +\overline{\gamma_n} y) \prod_{j=0}^{n-1} (\alpha_j x +\gamma_j y) e^{i \phi}
  +(-1)^n   (\alpha_n x +\gamma_n y) \prod_{j=0}^{n-1} (\overline{\alpha_j} x +\overline{\gamma_j} y) e^{-i \phi}\\
  ={} &  G_n(t,x,y)- |\alpha_n x +\gamma_n y|^2\ H_n(t,x,y) + 2 (-1)^n \Re\left( (\overline{\alpha_n} x +\overline{\gamma_n} y) \prod_{j=0}^{n-1} (\alpha_j x +\gamma_j y) e^{i \phi} \right).
  \end{align*}
Computing the real part of the last term above, we obtain the equation 
 \begin{align*}
F_{T_\phi}(t,x,y)={}& G_{n}(t, x, y) - |\alpha_n x +\gamma_n y|^2 H_{n}(t, x, y) 
  + 2(-1)^n \  \Re\left( \left(\overline{\alpha_n} x +\overline{\gamma_n} y\right) {\prod_{j=0}^{n-1}} (\alpha_j x +\gamma_j y)\right) \cos \phi \\ 
  & -2 (-1)^n \  
    \Im\left( \left(\overline{\alpha_n} x +\overline{\gamma_n} y\right) { \prod_{j=0}^{n-1}}  (\alpha_j x +\gamma_j y)\right) \sin \phi,     \end{align*}
which completes the proof.
\end{proof}

For every $n \in \N$ and for a fixed point $(x, y) \in {\mathbb R}^2$, the angle $\phi \in [0, 2\pi)$ is involved only in the constant term (with respect to the variable $t$) of the polynomial 
 $F_{T_\phi}(t,x,y)$. Furthermore, for every $(x, y) \in {\mathbb R}^2$ and for every $\phi \in [0, 2 \pi)$, the polynomial $F_{T_\phi}(t,x,y)$, seen as a polynomial in $t$, has $n+1$ real roots, counting multiplicities, as it is the characteristic polynomial of the Hermitian matrix $-x \Re(T_\phi)-y \Im(T_\phi)$. The following lemma will be useful later when applied to the polynomial $F_{T_\phi}$.

\begin{lemma}\label{le:maxmax}
Let $F(t:\phi)$ be a family of polynomials in $\R[t]$ given by the expression 
\[
F(t: \phi) =t^{n+1} +p_n t^n + \ldots +p_1 t +p_0 -u \cos \phi -v \sin \phi,
\]
where $\phi \in [0,2\pi)$. Assume that the polynomial $F(t: \phi)$ has $n+1$ real roots counting multiplicities for any angle $\phi \in [0, 2\pi)$. Let $\phi_0$, $\phi_1 \in [0, 2\pi)$ be such that
\[  u \cos \phi_0 +v \sin \phi_0 = - \sqrt{u^2 +v^2} \quad \text{ and } \quad 
u \cos \phi_1 +v\sin \phi_1 =  \sqrt{u^2 +v^2}.
\]
Then
\[
\max \left\{  \max \left\{t \in \R \, : \, F(t: \phi)=0 \right\} : 0 \leq \phi < 2 \pi \right\}
= \max \left\{t \in \R \, : \, F(t :  \phi_1)=0 \right\}, 
\]
and
\[
    \min\left\{ \max \left\{t \in \R \, : \, F(t: \phi)=0 \right\} : 0 \leq \phi < 2 \pi \right\} = \max \left\{t \in \R \, : \, F(t: \phi_0)=0 \right\}.
    \]
\end{lemma}
\begin{proof}
Define $p(t)$ as
\[
p(t)=t^{n+1} +p_n t^n + \ldots +p_1 t +p_0.
\]
Observe that, by assumption, the equation
\[
p(t)=u \cos \phi + v \sin \phi
\]
has $n+1$ real solutions (counting multiplicities) for every $\phi \in [0, 2 \pi)$. For some $\phi \in [0, 2 \pi)$, we have $u \cos \phi + v \sin \phi =0$, and hence $p$ has $n+1$ real roots (counting multiplicities) and the derivative of $p$ has $n$ real roots (counting multiplicities). Let $r_0$ be the largest root of $p'(t)$. Hence, $p$ is increasing on the interval $[r_0, \infty)$ and the equations \[
p(t)=u \cos \phi + v \sin \phi
\]
have a unique solution on the interval $[r_0, \infty)$.

Observe that for every $\phi\in [0, 2\pi)$
\[
- \sqrt{u^2 +v^2} \leq u \cos \phi +v \sin \phi \leq \sqrt{u^2 +v^2};
\]
equality occurs on the left-hand-side inequality at $\phi_0$ while equality occurs on the right-hand-side inequality at $\phi_1$.

For each $\phi \in[0, 2\pi)$, consider the number 
\[
\max\{ t \in \R \, : \, p(t)=u \cos \phi + v \sin \phi\}.
\]
Since the function $p$ is increasing on $[r_0, \infty)$, the largest of these numbers, when $\phi$ varies, occurs when $t$ is the largest solution of the equation
\[
p(t)= \sqrt{u^2 +v^2}.
\]
Hence we have 
\[
\max \left\{  \max \left\{t\in \R \, : \, F(t, \phi)=0 \right\} : 0 \leq \phi < 2 \pi \right\}
= \max \left\{t \in \R \, : \, F(t, \phi_1)=0 \right\}.
\]

Analogously, the smallest, when $\phi$ varies in $[0, 2\pi)$, among the largest solutions $t$ of the equations
\[
p(t) =u \cos \phi +v \sin \phi,
 \]
occurs when $t$ is the largest solution of the equation
\[
p(t)= -\sqrt{u^2 +v^2}.
\]
Hence we have 
\[
\min \left\{  \max \left\{t \in \R \, : \, F(t, \phi)=0 \right\} : 0 \leq \phi < 2 \pi \right\}
= \max \left\{t \in \R \, : \, F(t, \phi_0)=0 \right\}. \qedhere
\]
\end{proof}

In Theorem~\ref{2(n+1)}, we will show that the closure of the numerical range of $T$ is the numerical range of a single matrix. One of the key steps in the proof of said theorem will be to use the following proposition, which computes the closure of the numerical range of $T$ by using a single homogeneous polynomial, instead of the uncountable number of Kippenhahn polynomials of the symbols $T_\phi$,  which Theorem 2.8 in \cite{IM} would suggest: this is achieved by getting rid of the parameter $\phi$ in the expression of the Kippenhahn polynomial of the symbol $T_\phi$ in Proposition \ref{prop:kipp}.

\begin{proposition}\label{prop:defP}
Let $n \in \N$. Suppose that $T(a, b, c)$ is an $n+1$-periodic tridiagonal operator acting on $\ell^2({\mathbb N}_0)$. Let $G_n$ and $H_n$ be as in Proposition \ref{prop:kipp} and let $P$ be the real homogeneous polynomial of degree $2(n+1)$ given by
\begin{equation*}
   P(t, x, y)=\left( G_{n}(t, x, y)-|\alpha_n x +\gamma_n y|^2 H_{n}(t, x, y)\right)^2 
   -4 \prod_{j=0}^{n} \left|\alpha_j x +\gamma_j y\right|^2.
\end{equation*}
   Then $P(t, 0, 0)=t^{2(n+1)}$ and
  \[
  \sup \left\{\Re(e^{-i \theta} z)\colon\right. \bigl. z \in W\left(T(a, b,c)\right)\bigr\} 
   =\max \{ t \in {\mathbb R}\colon P(t, -\cos \theta, -\sin \theta) =0 \},
  \]
  for each $\theta \in [0, 2\pi)$.
\end{proposition}
\begin{proof}
It is trivial to check that $P(t,0,0)=t^{2(n+1)}$. Now, let $F(t:\phi)=F_{T_\phi}(t,x,y)$, where we know by Proposition \ref{prop:kipp} that \[
F_{T_\phi}(t,x,y)=G_n(t,x,y)-|\alpha_n x + \gamma_n y|^2 H_n(t,x,y) - u \cos \phi - v \sin \phi,
\]
where
\[
u= - 2(-1)^n \  \Re\left( \left(\overline{\alpha_n} x +\overline{\gamma_n} y\right) {\prod_{j=0}^{n-1}} (\alpha_j x +\gamma_j y)\right) 
\]
and 
\[
v= 2 (-1)^n \  
    \Im\left( \left(\overline{\alpha_n} x +\overline{\gamma_n} y\right) { \prod_{j=0}^{n-1}}  (\alpha_j x +\gamma_j y)\right).
    \]

Notice that
\begin{align*}
u^2+v^2= {} & 4 \Re^2\left( \left(\overline{\alpha_n} x +\overline{\gamma_n} y\right) {\prod_{j=0}^{n-1}} (\alpha_j x +\gamma_j y)\right) + 4 \Im^2\left( \left(\overline{\alpha_n} x +\overline{\gamma_n} y\right) { \prod_{j=0}^{n-1}}  (\alpha_j x +\gamma_j y)\right) \\
= {} & 4 \left| \left(\overline{\alpha_n} x +\overline{\gamma_n} y\right) {\prod_{j=0}^{n-1}} (\alpha_j x +\gamma_j y)\right|^2 \\
= & {} 4 \prod_{j=0}^{n} \left|\alpha_j x +\gamma_j y\right|^2.
\end{align*}

The polynomial $F(t: \phi)$ has the form outlined in Lemma~\ref{le:maxmax} and, as was mentioned before Lemma~\ref{le:maxmax}, it has $n+1$ real roots, counting multiplicities. Hence, by Lemma~\ref{le:maxmax}, for $\phi_0$ and $\phi_1$ satisfying
\[
u \cos(\phi_0)+v\sin(\phi_0)=-\sqrt{u^2+v^2}, \qquad u \cos(\phi_1)+v\sin(\phi_1)=\sqrt{u^2+v^2},
\]
we have that
\[
\max \left\{  \max \left\{t: F(t : \phi)=0 \right\} : 0 \leq \phi < 2 \pi \right\}
= \max \left\{t: F(t:\phi_1)=0 \right\}, 
\]
and
\[
    \min\left\{ \max \left\{t: F(t: \phi)=0 \right\} : 0 \leq \phi < 2 \pi \right\} = \max \left\{t: F(t: \phi_0)=0 \right\}.
    \]

Notice that
\begin{align*}
    F(t : \phi_0) \cdot F(t : \phi_1)
    = {} & \left(G_n(t,x,y)-|\alpha_n x + \gamma_n y|^2 H_n(t,x,y) - \left(u \cos(\phi_0) +v \sin(\phi_0)\right)\right) \\
    & \cdot \left(G_n(t,x,y)-|\alpha_n x + \gamma_n y|^2 H_n(t,x,y) - \left(u \cos(\phi_1) +v \sin(\phi_1)\right)\right)\\
    = {} & \left(G_n(t,x,y)-|\alpha_n x + \gamma_n y|^2 H_n(t,x,y) \right)^2 - \left(\sqrt{u^2+v^2}\right)^2 \\
    = {} & \left(G_n(t,x,y)-|\alpha_n x + \gamma_n y|^2 H_n(t,x,y) \right)^2 - (u^2 + v^2)\\
    ={} & P(t,x,y).
    \end{align*}

We also have, for each $\theta \in [0, 2\pi)$, that    
\begin{equation}\label{eq:TphiP}
    \begin{split}
\max\{ t \in {\mathbb R}\colon F_{T_\phi}(t,-\cos \theta,& -\sin \theta)=0,\ 0 \leq \phi < 2\pi \} \\
= {}& \max\left\{\max\left\{ t \in {\mathbb R}\colon F_{T_\phi}(t,-\cos \theta,-\sin \theta)=0 \right\} \colon \ 0 \leq \phi < 2\pi \right\} \\
= {}& \max\left\{ t \in {\mathbb R}\colon F_{T_{\phi_1}}(t,-\cos \theta,-\sin \theta)=0 \right\} \\
   = {} & \max \{ t \in {\mathbb R}\colon P(t, -\cos \theta, -\sin \theta) =0 \}.
   \end{split}
   \end{equation}
   The last equality follows since the roots of $P(t,-\cos \theta,-\sin \theta)$ are those of $F(t:\phi_1)=F_{T_{\phi_1}}(t,-\cos \theta,-\sin \theta)$ and $F(t:\phi_0)=F_{T_{\phi_0}}(t,-\cos\theta,-\sin \theta)$, so by the choice of $\phi_0$ and $\phi_1$, the largest root of $P(t,-\cos \theta,-\sin \theta)$ is the largest root of $F_{T_{\phi_1}}(t,-\cos \theta,-\sin \theta)$. 

By the definition of the Kippenhahn polynomial, we have
\[
\max \left\{ t \in {\mathbb R}\colon F_{T_\phi}(t,-\cos(\theta),-\sin(\theta))=0 \right\}
 =
\max \left\{ \Re(e^{-i \theta} z): z \in W(T_{\phi})\right\}.
\]
and hence we obtain
\begin{equation}\label{eq:Kip_Tphi}
\begin{split}
\max \left\{ t \in \R \, : \,  F_{T_\phi}(t,-\cos(\theta),-\sin(\theta))=0,\ 0 \leq \phi < 2\pi \right\}
 \\ = 
\max \left\{ \Re(e^{-i \theta} z)  \, : \,  z \in W(T_{\phi}), \  0 \leq \phi < 2\pi\right\}.
\end{split}
\end{equation}

Lastly, the equality 
\begin{equation}\label{eq:theoremIM}
  \sup \left\{\Re(e^{-i \theta} z) \, : \,  z \in W\left(T(a, b,c)\right)\right\} 
  =\max \left\{ \Re(e^{-i \theta} z) \, : \, z \in W\left(T_{\phi}\right), \ 0 \leq \phi < 2\pi\right\}
  \end{equation}
follows from Theorem 2.8 in \cite{IM}. Putting together equations \eqref{eq:TphiP}, \eqref{eq:Kip_Tphi} and \eqref{eq:theoremIM}, we obtain the desired conclusion.
\end{proof}

The following definition will be useful.

  \begin{definition} 
  Suppose that $Q(t, x, y)$ is a real homogeneous polynomial in $3$ variables $t, x, y$ of degree $m$ with 
  $Q(1, 0, 0) > 0$. If the equation 
   $Q(t, x_0, y_0) =0$ in $t$ has $m$ real solutions counting multiplicities for any $(x_0, y_0) \in {\mathbb R}^2$ with $x_0^2 +y_0^2 > 0$, we say that $Q$ is {\em hyperbolic} (with respect to $(1,0,0)$).
\end{definition}

 The above condition may  also be formulated as: ``the equation $Q(t, -\cos \theta, -\sin \theta) =0$ in $t$ has $m$ real solutions for any 
 angle $0 \leq \theta < 2\pi$''. 

  \begin{theorem}[Plaumann and Vinzant \cite{PV}] \label{PV}
  Suppose that $Q(t, x, y)$ is a real homogeneous hyperbolic polynomial of degree $m$
   with $Q(1,0,0)=1$. Then there exists an $m \times m$ complex matrix $A$ satisfying 
          $$Q(t, x, y) ={\rm det}(t I_m +x {\rm Re}(A) +y {\rm Im}(A)).$$
\end{theorem}

\smallskip
\noindent{\bf Remark.} Helton and Vinnikov \cite{HV} (cf. \cite{HS}) proved a result stronger than the above theorem which guarantees that we can construct 
an $m \times m$ complex {\em symmetric} matrix $A$ satisfying a similar property. In this paper we do not use the symmetry of the matrix $A$. 

\vskip5pt  Depending on the above Theorem~\ref{PV}, we obtain the main theorem of this paper.

   \begin{theorem}\label{2(n+1)} 
   Suppose that $T(a, b, c)$ is an $n+1$-periodic tridiagonal operator acting on $\ell^2({\mathbb N}_0)$. Then there exists a $2(n+1) \times 2(n+1)$ complex matrix $A$ such that
    \[
    \ovl{W(T(a, b,c))}=W(A)
    \]
where the matrix $A$ is chosen so that it satisfies
\begin{equation*}
   F_A(t, x, y)=\left( G_{n}(t, x, y)-|\alpha_n x +\gamma_n y|^2 H_{n}(t, x, y)\right)^2 
   -4 \prod_{j=0}^{n} \left|\alpha_j x +\gamma_j y\right|^2,
\end{equation*}
where the polynomials $G_n$ and $H_n$ are as in Proposition \ref{prop:kipp}.
\end{theorem}
\begin{proof}
By Theorem \ref{PV}, there exists a $2(n+1) \times 2(n+1)$ matrix $A$ such that $P(t,x,y)=F_A(t,x,y)$, where $P$ is the homogeneous polynomial in Proposition \ref{prop:defP}. But also, by the same proposition, 
  \begin{align*}
  \sup \left\{\Re(e^{-i \theta} z)\colon  z \in W\left(T(a, b,c)\right)\right\} 
   = {}& \max \{ t \in {\mathbb R}\colon F_A(t, -\cos \theta, -\sin \theta) =0 \} \\
   = {}&  \max \left\{\Re(e^{-i \theta} z)\colon  z \in W(A)\right\} 
  \end{align*}
for each $\theta \in [0, 2\pi)$, and hence the closure of the numerical range of $T(a,b,c)$ equals the numerical range of $A$.
\end{proof}

It is clear that given the operator $T$, one can compute the polynomial $P$ which, by the Plaumann-Vinzant Theorem, is the Kippenhahn polynomial of some matrix $A$. The question arises on whether the matrix $A$ can be explicitly computed. The paper \cite{PV} shows a method for constructing such a matrix $A$ (see also \cite{HV, PSV}). 

In some cases, the matrix $A$ can be found explicitly, as the next proposition shows. The reader should compare our next result to Theorem 4.1 in \cite{BebianoEtAl}, where an alternative method for computing the numerical range of the tridiagonal operator $T(a,b,c)$ is obtained, when $a$, $b$ and $c$ are real $2$-periodic sequences.

\begin{proposition}\label{prop:2b2}
  Let $a$ and $c$ be real $2$-periodic sequences and let $b$ the constant $0$ sequence. If
  \[
  S=\begin{pmatrix}
  \alpha_0 + \alpha_1 & -\gamma_0 & - \gamma_1 & 0 \\
  -\gamma_0 & -\alpha_0 + \alpha_1 & 0 & - \gamma_1 \\
  - \gamma_1 & 0 & \alpha_0 - \alpha_1 & - \gamma_0 \\
  0 & -\gamma_1 & - \gamma_0  & -\alpha_0-\alpha_1
  \end{pmatrix}
  \]
  then $\ovl{W(T(a,b,c))}=W(S)$.
\end{proposition}
\begin{proof}
It is a straightforward computation that the polynomial $P$ in Proposition~\ref{prop:defP} equals 
\[
P(t,x,y)=\left(t^2-|\alpha_0 x + \gamma_0 y |^2 - |\alpha_1 x + \gamma_1 y |^2 \right)^2 - 4 |\alpha_0 x + \gamma_0 y |^2 |\alpha_1 x + \gamma_1 y |^2
\]
But a computation also shows that $F_S(t,x,y)=P(t,x,y)$ and hence, by Theorem \ref{2(n+1)}, we have $\ovl{W(T(a,b,c))}=W(S)$.
\end{proof}

We illustrate the above proposition with some examples.

\begin{example}
  Let $a$ be the $2$-periodic sequence with period word $1\ 3$, let $b$ be the constant $0$ sequence and let $c$ be the $2$-periodic sequence with period word $4\ 8$. Then, by Proposition~\ref{prop:2b2}, if
  \[
  S=\begin{pmatrix}
  8 & \frac{1}{2} i  & -\frac{7}{2}i & 0 \\[5pt]
  \frac{1}{2}i & 1 & 0 & -\frac{7}{2}i \\[5pt]
  -\frac{7}{2}i & 0 & -1 &  \frac{1}{2}i\\[5pt]
  0 & -\frac{7}{2}i & \frac{1}{2}i & -8
  \end{pmatrix},
  \]
  then $\ovl{W(T(a,b,c))}=W(S)$. The boundary of the numerical range of $S$ is shown in Figure \ref{fig:irreducible}. The Kippenhahn polynomial of $S$ equals
  \[
P(t,x,y)= t^4 -65t^2 x^2-25 t^2 y^2 +64x^4 +192 x^2 y^2 + 144 y^4.
  \]

The  quartic curve $P(t, x, y) =0$  in the complex projective plane has a pair of ordinary singular  points  of multiplicity  $2$ at $(t, x, y)=(0, 1,  \pm  i  \sqrt{2/3})$ and  there is no other singular point. So the  algebraic curve theory tell us that the homogeneous polynomial $P(t, x, y)$ is  irreducible in the polynomial ring.

\begin{figure}
    \centering
    \includegraphics[scale=0.5]{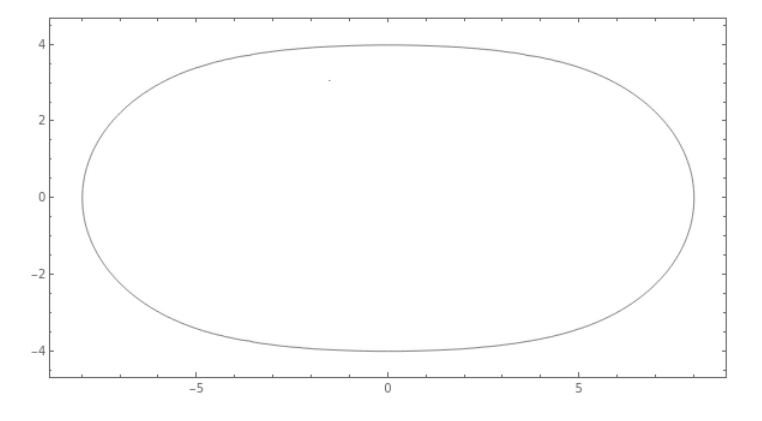}
    \vspace{-0.7cm}
    \caption{Boundary of the numerical range of $S$}
    \label{fig:irreducible}
\end{figure}

Hence, using for example Proposition 2.3 in \cite{GauWu}, there cannot be a matrix $R$ of size $m\times m$, with $1\leq m < 4$ with $W(R)=W(S)$. Incidentally, this shows that the size of the matrix $A$ in Theorem~\ref{2(n+1)} is optimal.
\end{example}

\begin{example}\label{ex:gamma1=0}
Let $a$ and $c$ be real $2$-periodic sequences with period words $a_0 a_1$ and $c_0 c_1$ respectively, and let $b$ be the constant sequence $0$. If $a_0=c_1$, then $\gamma_1=0$ and then, by Proposition~\ref{prop:2b2}, $\ovl{W(T(a, b,c))}=W(S)$, where
\[
S=\begin{pmatrix}
  \alpha_0 + \alpha_1 & -\gamma_0 & 0 & 0 \\
  -\gamma_0 & -\alpha_0 + \alpha_1 & 0 & 0 \\
  0 & 0 & \alpha_0 - \alpha_1 & - \gamma_0 \\
  0 & 0 & - \gamma_0  & -\alpha_0-\alpha_1
  \end{pmatrix}.
\]
But this implies that
\[
\ovl{W(T(a, b,c))}=\conv{W(A+\alpha_1 I) \cup W(A - \alpha_1 I)},
\]
where
\[
A=\begin{pmatrix}
\alpha_0  & - \gamma_0 \\
- \gamma_0 & -\alpha_0  
\end{pmatrix}.
\]
That is, $\ovl{W(T(a,b,c))}$ is the convex hull of two ellipses (possibly degenerate), each one a translation of a single elllipse (possibly degenarate) centered at the origin.
\end{example}

\begin{example}
  Let $a$ be the $2$-periodic sequence with period word $1,-1$, let $b$ be the constant $0$ sequence and let $c$ be the constant $1$ sequence. Then, 
 by Example~\ref{ex:gamma1=0}, we have that $\ovl{W(T(a,b,c))}=\conv{W(A+I) \cup W(A-I)}$, where
 \[
 A=\begin{pmatrix}
 0 & i  \\[5pt]
 i & 0  
\end{pmatrix}.
 \]
 But it is easy to see that $W(A)$ is the closed line segment joining $-i$ and $i$. Hence, $\ovl{W(T(a,b,c))}$ equals the convex hull of the closed line segment joining $-1-i$ and $-1+i$ and the closed line segment joining $1-i$ and $1+i$; i.e., the square with vertices $-1-i$, $-1+i$, $1-i$ and $1+i$, recovering (most of) Theorem 9 in \cite{CN2011}.
\end{example}

\begin{example}\label{ex:gamma0=0}
  Let $a$ and $c$ be real $2$-periodic sequences with period words $a_0 a_1$ and $c_0 c_1$ respectively, and let $b$ be the constant sequence $0$. If $c_0=a_1$, then $\gamma_0=0$ and then, by Proposition~\ref{prop:2b2}, $\ovl{W(T(a, b,c))}=W(S)$, where
\[
S=\begin{pmatrix}
  \alpha_0 + \alpha_1 & 0 & -\gamma_1 & 0 \\
  0 & -\alpha_0 + \alpha_1 & 0 & -\gamma_1 \\
  -\gamma_1 & 0 & \alpha_0 - \alpha_1 & 0 \\
  0 & - \gamma_1 & 0  & -\alpha_0-\alpha_1
  \end{pmatrix}.
\]
But if 
\[
U=\begin{pmatrix}
1 & 0 & 0 & 0\\
0 & 0 & 1 & 0\\
0 & 1 & 0 & 0\\
0 & 0 & 0 & 1
\end{pmatrix},
\]
then 
\[
U^* S U =\begin{pmatrix}
  \alpha_0 + \alpha_1 & -\gamma_1 & 0& 0 \\
  -\gamma_1 & \alpha_0 - \alpha_1 & 0 & 0 \\
  0 & 0 & -\alpha_0 + \alpha_1 &  -\gamma_1 \\
  0 & 0 & - \gamma_1 &  -\alpha_0-\alpha_1
  \end{pmatrix}.
\]

But this implies that
\[
\ovl{W(T(a, b,c))}=\conv{W(A+\alpha_0 I) \cup W(A - \alpha_0 I)},
\]
where
\[
A=\begin{pmatrix}
\alpha_1  & - \gamma_1\\
- \gamma_1 & -\alpha_1  
\end{pmatrix}.
\]
That is, $\ovl{W(T(a,b,c))}$ is the convex hull of two ellipses (possibly degenerate), each one a translation of a single elllipse (possibly degenarate) centered at the origin.
\end{example}


\begin{example}
  Let $a$ be the $2$-periodic sequence with period word $01$, let $b$ be the constant $0$ sequence and let $c$ be the constant $1$ sequence. Then, 
 by Example~\ref{ex:gamma0=0}, we have that $\ovl{W(T(a,b,c))}=\conv{W(A+I) \cup W(A-I)}$, where
 \[
 A=\begin{pmatrix}
\frac{1}{2} & - \frac{1}{2} i \\[5pt]
- \frac{1}{2} i  & -\frac{1}{2}  
\end{pmatrix}.
 \]
 But, if 
 \[
 U=\begin{pmatrix}
 \frac{1}{\sqrt{2}} &  \frac{1}{\sqrt{2}} \\[10pt]
 -\frac{1}{\sqrt{2}} i &  \frac{1}{\sqrt{2}} i 
 \end{pmatrix},
 \]
  then $U$ is unitary and $U^* A U$ equals
 \[
 \begin{pmatrix}
 0 & 1 \\
 0 & 0
 \end{pmatrix}.
 \]
 Therefore, $\ovl{W(T(a,b,c))}=\conv{W(B+I) \cup W(B-I)}$, recovering the result in \cite[Theorem 3.6]{IM}. 
 \end{example}

In the paper \cite{IMN} we explore some sufficient conditions under which the matrix $A$ can be explicitly found, namely if $b=0$ and there is some symmetry in the periodic sequences $a$ and $c$, then the polynomial $P$ can be factored as the product of the Kippenhahn polynomials of two computable matrices, which generalizes the previous four examples.
  
\bibliographystyle{plain}

\end{document}